\documentclass[12pt,reqno]{article}

\usepackage{amsfonts,amsthm,amsmath,amssymb}
\usepackage{tikz}
\usepackage{tikz-cd}
\usetikzlibrary{decorations.pathreplacing}
\usepackage{authblk}

\usepackage{amscd}

\usepackage[latin2]{inputenc}

\usepackage{t1enc}

\usepackage[mathscr]{eucal}

\usepackage{indentfirst}

\usepackage{graphicx}

\usepackage{graphics}

\usepackage{pict2e}

\usepackage{mathrsfs}

\usepackage{enumerate}

\usepackage[pagebackref]{hyperref}
\hypersetup{colorlinks=true}
\usepackage{cite}
\usepackage{color}
\usepackage{epic}
\usepackage{hyperref} %this gives clickable references, which is nice
\usepackage{bm} %for boldface greek letters
\numberwithin{equation}{section}
\theoremstyle{plain}

\newtheorem{theorem}{Theorem}[section]

\newtheorem{lemma}[theorem]{Lemma}

\newtheorem{proposition}[theorem]{Proposition}

\theoremstyle{definition}

\newtheorem{Def}[theorem]{Definition}

\newtheorem{example}[theorem]{Example}

\newtheorem{conjecture}[theorem]{Conjecture}

\newtheorem{remark}[theorem]{Remark}

\newtheorem{?}[theorem]{Problem}

\newcommand\catalannumber[3]{
  \fill[gray!25]  (#1) rectangle +(#2,#2);
  \draw[help lines] (#1) grid +(#2,#2);
  \draw[dashed] (#1) -- +(#2,#2);
  \coordinate (prev) at (#1);
  \foreach \dir in {#3}{
    \ifnum\dir=0
    \coordinate (dep) at (0,1);
    \else
    \coordinate (dep) at (1,0);
    \fi
    \draw[line width=2pt,-stealth] (prev) -- ++(dep) coordinate (prev);
  };
}

\def\boxit#1{\leavevmode\hbox{\vrule\vtop{\vbox{\kern.33333pt\hrule
    \kern1pt\hbox{\kern1pt\vbox{#1}\kern1pt}}\kern1pt\hrule}\vrule}}
\newcommand{\circo}{~\raisebox{1pt}{\tikz \draw[line width=0.5pt] circle(1.2pt);}~}% for smaller compositional circle

\def\fS{\mathfrak{S}}
\def\rA{\mathrm{A}}
\def\rD{\mathrm{D}}
\def\rP{\mathrm{P}}

\def\cP{\mathcal{P}}
\def\cS{\mathcal{S}}
\def\cN{\mathcal{M}}
\def\bPI{\mathbf{P_i}}

\def\bs{\mathbf{s}}
\def\bt{\mathbf{t}}

\newcommand\PGI[1]{\mathrm{P}_{i,#1}}

\def\st{\mathrm{st}}
\def\lrmin{\mathrm{lrmin}}

\newcommand{\pattern}[4]{			% mesh pattern
	\raisebox{0.6ex}{
		\begin{tikzpicture}[scale=0.35, baseline=(current bounding box.center), #1]
			\foreach \x/\y in {#4}		\fill[gray!50] (\x,\y) rectangle +(1,1);
			\draw (0.01,0.01) grid (#2+0.99,#2+0.99);
			\foreach \x/\y in {#3}		\filldraw (\x,\y) circle (6pt);
	\end{tikzpicture}}
}

\title{On fourteen equidistribution conjectures of Lv and Zhang and monotone mesh patterns with corner shadings}

\author[a]{Qi Fang\thanks{qifangpapers@163.com (corresponding author)}}
\author[a,b]{Shishuo Fu \thanks{fsshuo@cqu.edu.cn}}
\author[c]{Segey Kitaev \thanks{sergey.kitaev@strath.ac.uk}}
\author[a]{Haijun Li \thanks{lihaijun@cqu.edu.cn}}

\affil[a]{College of Mathematics and Statistics, Chongqing University, Chongqing 401331, PR China.}
\affil[b]{Key Laboratory of Nonlinear Analysis and its Applications (Chongqing University), Ministry of Education, Chongqing 401331, PR China.}
\affil[c]{Department of Mathematics and Statistics, University of Strathclyde, 26 Richmond Street, Glasgow G1 1XH, UK.}

\begin{document}
\maketitle	
	\begin{abstract}
Three complementation-like involutions are constructed on permutations to prove, and in some cases generalize, all remaining fourteen joint symmetric equidistribution conjectures of Lv and Zhang. Further enumerative results are obtained for several classes of (mesh) pattern-avoiding permutations, where the shadings of all involved mesh patterns are restricted to an opposing pair of corners. \\	

\end{abstract}

%%%%%%%%%%%%%%%%%%%%%%%%%%%%%%%%%%%%%%%%
\section{Introduction}\label{sec:intro}
%%%%%%%%%%%%%%%%%%%%%%%%%%%%%%%%%%%%%%%%

Denote the interval $\{x \in \mathbb{Z} : a \le x \le b\}$ of integers by $[a,b]$, and abbreviate $[1,n]$ as $[n]$. Let $\mathfrak{S}_n$ denote the symmetric group of all permutations of $[n]$. The empty permutation is denoted by $\varepsilon$. Given a permutation $\pi = \pi_1\pi_2\cdots\pi_m \in \mathfrak{S}_m$, we say that a permutation $\sigma = \sigma_1\sigma_2\cdots\sigma_n \in \mathfrak{S}_n$ \emph{contains} $\pi$ as a (classical) pattern if there exists a subsequence $\sigma_{i_1}\sigma_{i_2}\cdots\sigma_{i_m}$ that is \emph{order-isomorphic} to $\pi$, i.e., $\sigma_{i_a} < \sigma_{i_b}$ if and only if $\pi_a < \pi_b$ for all $1 \le a, b \le m$. Otherwise, $\sigma$ is said to \emph{avoid} $\pi$. For example, the permutation 435612 avoids the pattern 132. The set of all $\pi$-avoiding permutations of $[n]$ is denoted by $\mathfrak{S}_n(\pi)$. We also use the more general notation $\mathfrak{S}_n(P)$, where $P$ is a finite collection of patterns. A permutation belongs to $\mathfrak{S}_n(P)$ if and only if it avoids every pattern in $P$. A $\sigma_i$ is a {\em left-to-right minimum} (resp., {\em left-to-right maximum}) if $\sigma_i<\sigma_j$ (resp., $\sigma_i>\sigma_j$) for all $1\leq j<i$. A $\sigma_i$ is a {\em right-to-left minimum} (resp., {\em right-to-left maximum}) if $\sigma_i<\sigma_j$ (resp., $\sigma_i>\sigma_j$) for all  $i<j\leq n$. For example, the left-to-right maxima in 435612 are 4, 5, and 6.

A well-known result attributed to MacMahon~\cite{Mac60} and Knuth~\cite{Knu75} states that the cardinality $|\mathfrak{S}_n(\pi)|$ is given by the $n$-th Catalan number $C_n := \frac{1}{n+1} \binom{2n}{n}$ for any of the six permutations $\pi$ in $\mathfrak{S}_3$. In view of the basic operations on $\mathfrak{S}_n$ such as reversal, complementation, and inversion, it is natural to treat $\mathfrak{S}_n(123)$ and $\mathfrak{S}_n(132)$ as representative cases. Several bijections have been constructed between $\mathfrak{S}_n(123)$ and $\mathfrak{S}_n(132)$ to provide direct combinatorial explanations for their equinumerosity; see~\cite{CK08} for a survey and detailed analysis.

Over the past two decades, there has been a significant increase in interest in the study of permutation patterns, which has grown into a remarkably prolific branch of enumerative combinatorics; see, for example, the introductory book~\cite{Kit11} and the references therein. In addition to the classical permutation patterns mentioned above, various generalizations to other types of patterns have continually emerged in the literature~\cite[Chapters 5, 6, 7]{Kit11}. The specific type of pattern that concerns us here is called a \emph{mesh pattern}, first introduced by Br\"and\'en and Claesson~\cite{BC11}. A mesh pattern is a pair $p = (\pi, R)$, where $\pi = \pi_1 \pi_2 \cdots \pi_m \in \mathfrak{S}_m$ and $R \subseteq [0,m] \times [0,m]$. It can be represented pictorially using a grid diagram decorated with $m$ heavy dots and $|R|$ shaded unit cells. The dots are placed at coordinates $(i, \pi_i)$ for $1 \le i \le m$, and a unit cell is shaded if and only if its bottom-left corner $(j,k)$ belongs to $R$. Ignoring the shadings and focusing on the relative heights of the dots recovers the familiar notion of a permutation matrix (see~\cite[Section~1.5]{Sta11}) for the underlying permutation $\pi$. For example, the mesh pattern $p = (132, \{(0,0), (2,3), (3,1), (3,2)\})$ is illustrated below:
$$
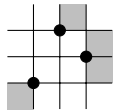
\pattern{scale=1}{3}{1/1,2/3,3/2}{0/0,2/3,3/1,3/2}
$$
Two patterns $p$ and $p'$ are said to be \emph{Wilf-equivalent}, denoted $p \sim p'$, if $|\mathfrak{S}_n(p)| = |\mathfrak{S}_n(p')|$ for all positive integers $n$.

One fruitful line of research on mesh patterns involves fixing a certain underlying permutation, say, $123$ (monotone) or $132$, and associating it with various kinds of shadings to create mesh patterns. Researchers then explore Wilf-equivalences among different patterns or compute the distribution of a fixed mesh pattern over all permutations, treating the number of occurrences of the chosen pattern as a permutation statistic. For example, in the work of Kitaev and Liese~\cite{KL13}, the distribution of certain ``border mesh patterns'' is linked to the \emph{harmonic numbers}, while two mesh patterns, each with precisely one shaded cell, are shown to have the \emph{Catalan distribution} on 132-avoiding permutations.

A systematic study of the distributions of mesh patterns has been carried out over the past few years in~\cite{KZ19,KZZ20,LK25,LZ24,LZ25}, extending the earlier systematic investigation of the avoidance of mesh patterns of length~2 in~\cite{HJS15}. In particular, the works of Lv and Kitaev~\cite{LK25} and Lv and Zhang~\cite{LZ24,LZ25} explore many (joint) equidistributions of mesh patterns with underlying permutations $123$, $321$, and $132$. The focus of \cite{LK25} was on symmetric shadings, while \cite{LZ24,LZ25} explored antipodal shadings as well.

In the present work, we continue to study mesh patterns $p=(\pi,R)$ whose underlying permutation pattern $\pi$ is either monotone or equal to $132$. More precisely, the mesh patterns in Figure~\ref{mesh-patterns-pic} are of interest to us. Note that we shall abbreviate as $\rA:=\rA_2$ and $\rD:=\rD_2$.

\begin{figure}[h]
\begin{align*}
	& \rP_1=\pattern{scale=1}{3}{1/1,2/2,3/3}{0/0,0/1,2/0,2/1,2/2,2/3,3/0,3/1,3/3}\ ,\quad 
	\rP_2=\pattern{scale=1}{3}{1/1,2/3,3/2}{0/0,0/1,2/0,2/1,2/2,2/3,3/0,3/1,3/3}\ ,\quad
	\rP_3=\pattern{scale=1}{3}{1/1,2/2,3/3}{0/0,0/1,2/0,2/1,2/3,3/0,3/1,3/3}\ ,\quad 
	\rP_4=\pattern{scale=1}{3}{1/1,2/3,3/2}{0/0,0/1,2/0,2/1,2/3,3/0,3/1,3/3}\ ,\\
	& \rP_5=\pattern{scale=1}{3}{1/1,2/2,3/3}{1/0,1/1,2/0,2/1,2/2,2/3,3/0,3/1,3/3}\ ,\quad
	\rP_6=\pattern{scale=1}{3}{1/1,2/3,3/2}{1/0,1/1,2/0,2/1,2/2,2/3,3/0,3/1,3/3}\ ,\quad
	\rP_7=\pattern{scale=1}{3}{1/1,2/2,3/3}{1/0,1/1,2/0,2/1,2/3,3/0,3/1,3/3}\ ,\quad
	\rP_8=\pattern{scale=1}{3}{1/1,2/3,3/2}{1/0,1/1,2/0,2/1,2/3,3/0,3/1,3/3}\ ,\\
	& \rP_9=\pattern{scale=1}{3}{1/1,2/2,3/3}{0/1,0/2,1/1,1/2,3/0,3/1,3/2,3/3}\ ,\quad
	\rP_{10}=\pattern{scale=1}{3}{1/1,2/3,3/2}{0/1,0/2,1/1,1/2,3/0,3/1,3/2,3/3}\ , \quad
	\rP_{11}=\pattern{scale=1}{3}{1/1,2/2,3/3}{1/0,0/2,1/1,1/2,3/0,3/1,3/2,3/3}\ ,\quad
	\rP_{12}\pattern{scale=1}{3}{1/1,2/3,3/2}{1/0,0/2,1/1,1/2,3/0,3/1,3/2,3/3}\ , \\
	&
	\rP_{13}=\pattern{scale=1}{3}{1/1,2/2,3/3}{1/0,1/1,2/0,2/1,2/2,3/0,3/1,3/2}\ ,\quad
	\rP_{14}=\pattern{scale=1}{3}{1/1,2/3,3/2}{1/0,1/1,2/0,2/1,2/2,3/0,3/1,3/2}\ ,\quad
	\rA_k=\raisebox{0.6ex}{
		\begin{tikzpicture}[scale=0.3, baseline=(current bounding box.center)]
			\fill[gray!50] (0,4.5) rectangle +(1,1);
			\fill[gray!50] (4.5,0) rectangle +(1,1);
			\filldraw (1,1) circle (6pt);
			\filldraw (2,2) circle (6pt);
			\filldraw (4.5,4.5) circle (6pt);
			\node[anchor=south] at (3,-2) {\tiny{$k$ heavy dots}};
			\node[anchor=center, rotate=45] at (3.25,3.25) {$\boldsymbol{\ldots}$};
			\draw [decorate, decoration={brace, mirror}] (0,-0.3) -- (6,-0.3);
			\draw [line width=0.4pt](0.01,1.01) -- (5.49,1.01); 
			\draw [line width=0.4pt](0.01,1.99) -- (5.49,1.99);
			\draw [line width=0.4pt](1.01,0.01) -- (1.01,5.49);
			\draw [line width=0.4pt](1.99,0.01) -- (1.99,5.49);
			\draw [line width=0.4pt](4.49,0.01) -- (4.49,5.49);
			\draw [line width=0.4pt](0.01,4.49) -- (5.49,4.49);
	\end{tikzpicture}},\quad 
	\rD_k=\raisebox{0.6ex}{
		\begin{tikzpicture}[scale=0.3, baseline=(current bounding box.center)]
			\fill[gray!50] (0,0) rectangle +(1,1);
			\fill[gray!50] (4.5,4.5) rectangle +(1,1);
			\filldraw (1,4.5) circle (6pt);
			\filldraw (2,3.5) circle (6pt);
			\filldraw (4.5,1) circle (6pt);
			\node[anchor=south] at (3,-2) {\tiny{$k$ heavy dots}};
			\node[anchor=center, rotate=135] at (3.25,2.25) {$\boldsymbol{\ldots}$};
			\draw [decorate, decoration={brace, mirror}] (0,-0.3) -- (6,-0.3);
			\draw [line width=0.4pt](0.01,1.01) -- (5.49,1.01); 
			\draw [line width=0.4pt](0.01,3.51) -- (5.49,3.51);
			\draw [line width=0.4pt](1.01,0.01) -- (1.01,5.49);
			\draw [line width=0.4pt](1.99,0.01) -- (1.99,5.49);
			\draw [line width=0.4pt](4.49,0.01) -- (4.49,5.49);
			\draw [line width=0.4pt](0.01,4.49) -- (5.49,4.49);
	\end{tikzpicture}}.
\end{align*}
\caption{The mesh patterns of interest in this paper}\label{mesh-patterns-pic}
\end{figure}
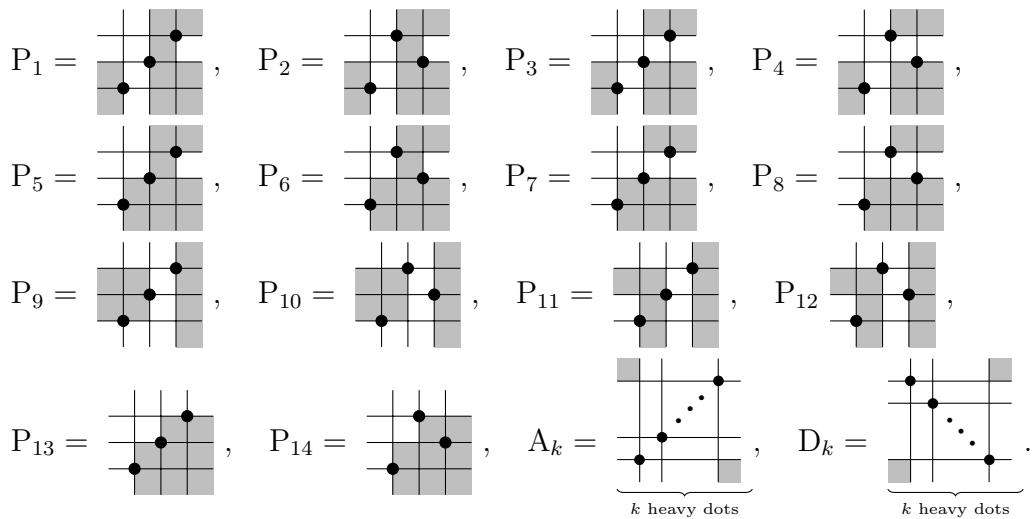

Our interest in the first 14 mesh patterns $\mathrm{P}_1$--$\mathrm{P}_{14}$ stems from the following seven joint equidistribution conjectures due to Lv and Zhang~\cite[Tab.~11]{LZ25}, where $k$ and $\ell$ are nonnegative integers. In this paper, we resolve all the conjectures, thereby completing the classification of joint equidistributions between the pairs of mesh patterns in question. A total of $112$ jointly equidistributed pairs had already been established in \cite{LZ25}, and we prove that all remaining speculative pairs are indeed jointly equidistributed. We also note that \cite{LZ25} raised further conjectures regarding equidistributions between different pairs of mesh patterns. For any pattern $\mathrm{P}$, we use $\mathrm{P}(\sigma)$ to denote the number of occurrences of $\mathrm{P}$ in a permutation $\sigma$.

\begin{conjecture}[{\cite[Conj.~1]{LZ25}}]\label{conj-1}
For $n\ge 1$ and any $k,\ell\geq 0$, we have
\begin{align*}
|\{\sigma\in\fS_n: \rP_1(\sigma)=k,~\rP_2(\sigma)=\ell\}| &= |\{\sigma\in\fS_n: \rP_2(\sigma)=k,~\rP_1(\sigma)=\ell\}|.
\end{align*}
\end{conjecture}

\begin{conjecture}[{\cite[Conj.~2]{LZ25}}]\label{conj-2}
For $n\ge 1$ and any $k,\ell\geq 0$, we have
\begin{align*}
|\{\sigma\in\fS_n: \rP_3(\sigma)=k,~\rP_4(\sigma)=\ell\}| &= |\{\sigma\in\fS_n: \rP_4(\sigma)=k,~\rP_3(\sigma)=\ell\}|.
\end{align*}
\end{conjecture}

\begin{conjecture}[{\cite[part of Conj.~3]{LZ25}}]\label{conj-3}
For $n\ge 1$ and and any $k,\ell\geq 0$, we have
\begin{align*}
|\{\sigma\in\fS_n: \rP_5(\sigma)=k,~\rP_6(\sigma)=\ell\}| &= |\{\sigma\in\fS_n: \rP_6(\sigma)=k,~\rP_5(\sigma)=\ell\}|.
\end{align*}
\end{conjecture}

\begin{conjecture}[{\cite[part of Conj.~4]{LZ25}}]\label{conj-4}
For $n\ge 1$ and any $k,\ell\geq 0$, we have
\begin{align*}
|\{\sigma\in\fS_n: \rP_7(\sigma)=k,~\rP_8(\sigma)=\ell\}| &= |\{\sigma\in\fS_n: \rP_8(\sigma)=k,~\rP_7(\sigma)=\ell\}|.
\end{align*}
\end{conjecture}

\begin{conjecture}[{\cite[part of Conj.~4]{LZ25}}]\label{conj-5}
For $n\ge 1$ and any $k,\ell\geq 0$, we have
\begin{align*}
|\{\sigma\in\fS_n: \rP_9(\sigma)=k,~\rP_{10}(\sigma)=\ell\}| &= |\{\sigma\in\fS_n: \rP_{10}(\sigma)=k,~\rP_9(\sigma)=\ell\}|.
\end{align*}
\end{conjecture}

\begin{conjecture}[{\cite[part of Conj.~4]{LZ25}}]\label{conj-6}
	For $n\ge 1$ and and any $k,\ell\geq 0$, we have
	\begin{align*}
		|\{\sigma\in\fS_n: \rP_{11}(\sigma)=k,~\rP_{12}(\sigma)=\ell\}| &= |\{\sigma\in\fS_n: \rP_{12}(\sigma)=k,~\rP_{11}(\sigma)=\ell\}|.
	\end{align*}
\end{conjecture}

\begin{conjecture}[{\cite[part of Conj.~4]{LZ25}}]\label{conj-7}
	For $n\ge 1$ and and any $k,\ell\geq 0$, we have
	\begin{align*}
		|\{\sigma\in\fS_n: \rP_{13}(\sigma)=k,~\rP_{14}(\sigma)=\ell\}| &= |\{\sigma\in\fS_n: \rP_{14}(\sigma)=k,~\rP_{13}(\sigma)=\ell\}|.
	\end{align*}
\end{conjecture}

Note that the shadings in none of the above seven patterns are symmetric, meaning that when we flip the picture of a pattern $\rP$ with respect to the diagonal $y=x$, we get a different mesh pattern which we denote as $\overline{\rP}$. It is worth noting that in Lv and Zhang's paper \cite{LZ25}, there are actually fourteen pairs of patterns conjectured to be jointly equidistributed. But proving the above seven conjectures is sufficient, since the validity of, say Conjecture~\ref{conj-1}, implies the truth of the following equidistribution via the operation of taking inverse $\sigma\mapsto\sigma^{-1}$:
\begin{align*}
|\{\sigma\in\fS_n: \overline{\rP}_1(\sigma)=k,~\overline{\rP}_2(\sigma)=\ell\}| &= |\{\sigma\in\fS_n: \overline{\rP}_2(\sigma)=k,~\overline{\rP}_1(\sigma)=\ell\}|.
\end{align*}
Here the patterns $\overline{\rP}_1$ and $\overline{\rP}_2$ are recorded together as the pair $Y_7^{(4)}$ in \cite[Tab.~11]{LZ25}.

Next, for our work related to the patterns $\rA_k$ and $\rD_k$, initially we intended to consider monotone mesh patterns with all four corners shaded as shown below:
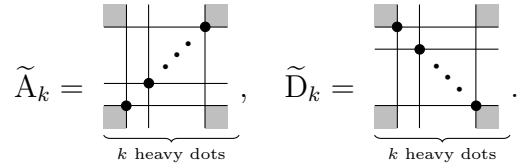
\begin{align*}
\widetilde{\rA}_k=\raisebox{0.6ex}{
		\begin{tikzpicture}[scale=0.3, baseline=(current bounding box.center)]
			\fill[gray!50] (0,4.5) rectangle +(1,1);
			\fill[gray!50] (4.5,0) rectangle +(1,1);
			\fill[gray!50] (0,0) rectangle +(1,1);
			\fill[gray!50] (4.5,4.5) rectangle +(1,1);
			\filldraw (1,1) circle (6pt);
			\filldraw (2,2) circle (6pt);
			\filldraw (4.5,4.5) circle (6pt);
			\node[anchor=south] at (3,-2) {\tiny{$k$ heavy dots}};
			\node[anchor=center, rotate=45] at (3.25,3.25) {$\boldsymbol{\ldots}$};
			\draw [decorate, decoration={brace, mirror}] (0,-0.3) -- (6,-0.3);
			\draw [line width=0.4pt](0.01,1.01) -- (5.49,1.01); 
			\draw [line width=0.4pt](0.01,1.99) -- (5.49,1.99);
			\draw [line width=0.4pt](1.01,0.01) -- (1.01,5.49);
			\draw [line width=0.4pt](1.99,0.01) -- (1.99,5.49);
			\draw [line width=0.4pt](4.49,0.01) -- (4.49,5.49);
			\draw [line width=0.4pt](0.01,4.49) -- (5.49,4.49);
	\end{tikzpicture}},\quad 
\widetilde{\rD}_k=\raisebox{0.6ex}{
		\begin{tikzpicture}[scale=0.3, baseline=(current bounding box.center)]
			\fill[gray!50] (0,0) rectangle +(1,1);
			\fill[gray!50] (4.5,4.5) rectangle +(1,1);
			\fill[gray!50] (0,4.5) rectangle +(1,1);
			\fill[gray!50] (4.5,0) rectangle +(1,1);
			\filldraw (1,4.5) circle (6pt);
			\filldraw (2,3.5) circle (6pt);
			\filldraw (4.5,1) circle (6pt);
			\node[anchor=south] at (3,-2) {\tiny{$k$ heavy dots}};
			\node[anchor=center, rotate=135] at (3.25,2.25) {$\boldsymbol{\ldots}$};
			\draw [decorate, decoration={brace, mirror}] (0,-0.3) -- (6,-0.3);
			\draw [line width=0.4pt](0.01,1.01) -- (5.49,1.01); 
			\draw [line width=0.4pt](0.01,3.51) -- (5.49,3.51);
			\draw [line width=0.4pt](1.01,0.01) -- (1.01,5.49);
			\draw [line width=0.4pt](1.99,0.01) -- (1.99,5.49);
			\draw [line width=0.4pt](4.49,0.01) -- (4.49,5.49);
			\draw [line width=0.4pt](0.01,4.49) -- (5.49,4.49);
	\end{tikzpicture}}.
\end{align*}
The following observation shows that, from the perspective of pattern avoidance (though not from that of distribution), it suffices to consider $\rA_k$ and $\rD_k$ instead, where only the opposite pair of corners are shaded.

\begin{proposition}\label{prop:2 corners suffice}
For any integers $n\ge k>1$, we have
\begin{align}
\fS_n(\rA_k) &= \fS_n(\widetilde{\rA}_k),\label{eq:A=tildeA}\\
\fS_n(\rD_k) &= \fS_n(\widetilde{\rD}_k).\label{eq:D=tildeD}
\end{align}
\end{proposition}

For the reader's convenience, our enumerative results related to the avoidance of the patterns \( \rA_k \), \( \rD_k \), and 132 are summarized in Table~\ref{tab:Enum}.

The rest of the paper is organized as follows. Conjectures~\ref{conj-1}--\ref{conj-7} are proved in Section~\ref{sec:conj}, which is divided into three subsections. The first subsection features our first involution $\Phi$ which proves a stronger result (see Theorem~\ref{thm:5-tuple}) that implies simultaneously the first four Conjectures \ref{conj-1}--\ref{conj-4}. We construct our second involution $\Psi$ in the second subsection to confirm and strengthen Conjectures \ref{conj-5} and \ref{conj-6} (see Theorem~\ref{thm:quadruple}). The last Conjecture~\ref{conj-7} is proven in two ways in the third subsection: first through a third involution $\theta$ which utilizes the {\it Lehmer code} to make transitions between permutations and {\it subexceedant functions}, and then by a recurrence relation for the corresponding bivariate generating function. Next, in Section~\ref{sec:monotone}, we prove Proposition~\ref{prop:2 corners suffice} and enumerate two classes of permutations that simultaneously avoid two monotone mesh patterns. Then, in Section~\ref{sec:132 and one monotone}, we focus on $132$-avoiding permutations and additionally consider one monotone mesh pattern to be avoided, obtaining five more enumerative results. It is worth noting that in Sections~\ref{sec:monotone} and~\ref{sec:132 and one monotone}, we restrict ourselves to cases where the mesh patterns involved have shadings only in an opposing pair of corners (recall Proposition~\ref{prop:2 corners suffice}). We conclude the paper with an outlook on future research, presented in Section~\ref{sec:concluding_remarks}.

\begin{table}[ht!]
	\renewcommand{\arraystretch}{1.6}
	\centering
	\begin{tabular}{|c||c||c|}
		\hline
		$p_1,p_2$ & enumeration formula for $|\fS_n(p_1,p_2)|$ & Ref.\\
		\hline
		\hline
		$\rA,\rD$ & $0 ,\; n\geq 2$ & Prop \ref{thm:P1P2} \\
		\hline 
		$\rA,\rD_3$ & $n-1 ,\; n\geq 2$ & Thm \ref{thm:P1P4} \\
		\hline 
		$\rA_3,\rD_3$ & $2^{n-1} ,\; n\geq 1$ & Thm \ref{thm:P3P4} \\
		\hline 
		$132,\rA$ & $C_{n}-C_{n-1} ,\; n\geq 2$ & Thm \ref{thm:132P1} \\
		\hline 
		% $132,\rA_3$ & $C_{n}-C_{n-1}+1 ,\; n\geq 3$ & Thm \ref{thm:132P3} \\
		% \hline 
		$132,\rA_k$ & $C_{n}-C_{n-1}+|\fS_{n-1}(132,12\ldots(k-1))|,\; n\geq 1$ & Thm \ref{thm:132Ak} \\
		\hline 
		$132,\rD$ & $C_{n-1} ,\; n\geq 1$ & Thm \ref{thm:132D} \\
		\hline 
		$132,\rD_3$ & $C_{n-1}+n-1,\; n\geq 1$ & Thm \ref{thm:132D3} \\
		\hline 
		$132,\rD_4$ & $C_{n-1}+(n-1)(2n^2-7n+12)/6,\; n\geq 1$ & Thm \ref{thm:132D4} \\
		\hline 
	\end{tabular}
	\medskip
	\caption{Enumeration results for $(p_1,p_2)$-avoiding permutations}
	\label{tab:Enum}
\end{table}

%%%%%%%%%%%%%%%%%%%%%%%%%%%%%%%%%%%%
\section{Lv and Zhang's conjectures}\label{sec:conj}
%%%%%%%%%%%%%%%%%%%%%%%%%%%%%%%%%%%%

We break down the proofs of all seven conjectures into three subsections. For the first four conjectures, we construct an involution that is essentially a variant of the complementation map. As it turns out, this involution actually leads to an equidistribution result that is stronger than (and thus generalizes) the original conjectures (see Theorem~\ref{thm:5-tuple}).

\subsection{Generalizations of Conjectures~\ref{conj-1}--\ref{conj-4}}
We begin by introducing the following slightly more general version of the complementation map.
\begin{Def}\label{def:comp}
	Let $\alpha:=a_1a_2\cdots a_m$ be a sequence of distinct positive integers and let $b_1b_2\cdots b_m$ be its unique rearrangement that is monotonely increasing. Then the {\it complement} of $\alpha$, denoted $\kappa(\alpha)$, is defined to be the sequence $\kappa(\alpha):=\hat{a}_1\hat{a}_2\cdots\hat{a}_m$, where for $1\le i\le m$, $\hat{a}_i=b_{m+1-j}$ if $a_i=b_j$. For example, $\kappa(4189)=8941$.
\end{Def}

The following definition and its alternative characterization (Lemma~\ref{lem:pattern char}) are crucial in understanding the eight patterns $\rP_1$--$\rP_8$ involved in the four conjectures.
\begin{Def}\label{def:active pair}
Given a mesh pattern $\rP_i$, $1\le i\le 8$, and a permutation $\sigma\in\fS_n$, a pair of letters $(x,y)$ is called {\it $\rP_i$-active} if there exists a triple subsequence $(\sigma_a,\sigma_b,\sigma_c)$, $1\leq a<b<c\leq n$, constituting an occurrence of $\rP_i$ in $\sigma$ with $x=\sigma_b$ and $y=\sigma_c$. 
\end{Def}

\begin{lemma}\label{lem:pattern char}
Given a mesh pattern $\rP_i$, $1\le i\le 8$, and a permutation $\sigma\in\fS_n$, then $(x,y)$ is a $\rP_i$-active pair if and only if 
\begin{itemize}
	\item for $i=1,5$, $x=\sigma_j>1$ is a right-to-left minimum and $y=\sigma_{j+1}$ is a right-to-left maximum for a certain $j$;
	\item for $i=3,7$, $x=\sigma_j>1$ is a right-to-left minimum, $y=\sigma_{k}$ is a right-to-left maximum for certain $1\le j<k\le n$, and $\sigma_j<\sigma_{\ell}<\sigma_k$ for any $\ell$, $j<\ell<k$;
	\item for $i=2,6$, $x=\sigma_j$ is a right-to-left maximum and $y=\sigma_{j+1}>1$ is a right-to-left minimum for a certain $j$;
	\item for $i=4,8$, $x=\sigma_j$ is a right-to-left maximum, $y=\sigma_{k}>1$ is a right-to-left minimum for certain $1\le j<k\le n$, and $\sigma_k<\sigma_{\ell}<\sigma_j$ for any $\ell$, $j<\ell<k$.
\end{itemize}
Moreover, counting the number of occurrences of the $\rP_i$-pattern in $\sigma$ is equivalent to counting the number of $\rP_i$-active pairs. Namely, we have
\begin{align}
\label{id:pair=count-1}
&\text{For $i=1,2,5,6$, }\quad \rP_i(\sigma)=|\{j\in[n]:(\sigma_j,\sigma_{j+1}) \text{ is $\rP_i$-active }\}|.\\
&\text{For $i=3,4,7,8$, }\quad \rP_i(\sigma)=|\{(j,k)\in[n]\times[n]:(\sigma_j,\sigma_{k}) \text{ is $\rP_i$-active }\}|.\label{id:pair=count-2}
\end{align}
\end{lemma}
\begin{proof}
We shall give proof details for only two cases $i=3$ and $i=8$, since all remaining cases could be proven analogously. Note that the cases $i=1,2,5,6$ have the added requirement that the two entries $x$ and $y$ in the pair must be adjacent, hence could be viewed as special cases of $i=3,4,7,8$, respectively.

Now suppose $i=3$ and we focus on pattern $\rP_3$. It is clear that $x=\sigma_j>1$ being a right-to-left minimum, $y=\sigma_k$ being a right-to-left maximum, and $\sigma_j<\sigma_\ell<\sigma_k$ for every $\ell$, $j<\ell<k$, are necessary conditions for $(x,y)$ to be $\rP_3$-active. To show these conditions are also sufficient and to establish \eqref{id:pair=count-2}, it suffices to show that for such a pair $(x,y)$ (which satisfies the aforementioned conditions), there exists a unique entry $z$, such that $(z,x,y)$ becomes a $\rP_3$-pattern in $\sigma$, making $(x,y)$ a $\rP_3$-active pair. Indeed, we define $z$ to be the leftmost entry that is smaller than $x$. Such a $z$ must exist since the conditions on $(x,y)$ force letter $1$ to be on the left of $x$, so the set $\{s\in[n]: s<j \text{ and } \sigma_s<\sigma_j=x\}$ is nonempty.

The case for $i=8$ can be explained along the same lines, except that the uniquely found $z$ which makes $(x,y)$ a $\rP_8$-active pair is now defined to be the rightmost (but still to the left of $\sigma_j$) entry that is smaller than $y=\sigma_k>1$.
\end{proof}

Comparing the two patterns $\rP_1$ and $\rP_3$, one notices that the only difference is whether or not the cell between the second and the third heavy dots is shaded. When this cell is shaded as in $\rP_1$, the last two entries (i.e., the $\rP_1$-active pair) must be adjacent in any occurrence of $\rP_1$. When this cell is not shaded as in $\rP_3$, then there could be any number of letters inbetween the $\rP_3$-active pair. The same observation can be made for the pattern pairs $(\rP_2,\rP_4)$, $(\rP_5,\rP_7)$, and $(\rP_6,\rP_8)$. The discussion above motivates the following refined notion of mesh pattern.

\begin{Def}
Given a permutation $\sigma\in\fS_n$ and a mesh pattern $\rP_i$, $i=3,4,7,8$, a triple $(\sigma_a,\sigma_b,\sigma_c)$, $1\le a < b < c \le n$, is said to form a $k$-refined $\rP_i$-pattern of $\sigma$, denoted as $\rP_{i,k}$, if $(\sigma_a,\sigma_b,\sigma_c)$ is a $\rP_i$ pattern and $c-b=k+1$. Introduce the $\rP_i$-pattern vector enumerator as
\begin{align*}
\bPI(\sigma) &:= (\PGI{0}(\sigma),\PGI{1}(\sigma),\ldots,\PGI{n-3}(\sigma)).
\end{align*}
\end{Def}

The involution we are going to construct actually proves the following much stronger result. We abbreviate the $m$-tuple of statistics as $(\st_1,\ldots,\st_m) \sigma:=(\st_1(\sigma),\ldots,\st_m(\sigma))$, where some of the statistics are allowed to be vector-valued. Recall the {\em Stirling statistic} $\lrmin(\sigma)$, which denotes the number of left-to-right minima in $\sigma$.

\begin{theorem}\label{thm:5-tuple}
For $n\ge 1$, the following quintuple identity holds for any permutation $\sigma\in\fS_n$ and its image $\hat{\sigma}:=\Phi(\sigma)$.
\begin{align}
\label{id:5-tuple}
(\lrmin,\mathbf{P_3},\mathbf{P_4},\mathbf{P_7},\mathbf{P_8}) \sigma &=(\lrmin,\mathbf{P_4},\mathbf{P_3},\mathbf{P_8},\mathbf{P_7}) \hat{\sigma}.
\end{align}
\end{theorem}

\begin{remark}
For $i=3,4,7,8$, we see that the pattern $\rP_{i,0}$ is precisely the pattern $\rP_{i-2}$, and that for any permutation $\sigma\in\fS_n$, we have
$$\sum_{k=0}^{n-3}\PGI{k}(\sigma)=\rP_i(\sigma).$$
Consequently, Theorem~\ref{thm:5-tuple} immediately confirms the first four Conjectures \ref{conj-1}--\ref{conj-4}.
\end{remark}

\begin{proof}[Proof of Theorem~\ref{thm:5-tuple}]
We construct an involution $\Phi:=\Phi_n$ over $\fS_n$ which proves \eqref{id:5-tuple}. Suppose a permutation $\sigma \in \fS_n$ decomposes as $\sigma=\sigma'1\sigma''$, where juxtaposition means concatenation. Its image is taken to be $\Phi(\sigma)=\hat{\sigma}:=\sigma'1\kappa(\sigma'')$, which is still a permutation in $\fS_n$ since the complementation $\kappa$ (see Definition~\ref{def:comp}) only rearranges the letters contained in $\sigma''$. 

First note that the involution $\Phi$ preserves the prefix of $\sigma$ that ends with $1$, so in particular, the number of left-to-right minima of $\sigma$ is preserved as well. Next, we focus on the quadruple of vector-valued pattern statistics, $(\mathbf{P_3},\mathbf{P_4},\mathbf{P_7},\mathbf{P_8})$. The key point is that upon taking the complementation map $\kappa$, a right-to-left maximum (resp., minimum) becomes a right-to-left minimum (resp., maximum). More precisely, take one occurrence (if any) of pattern $\rP_i$ ($i=3,4,7,8$), say $(\sigma_a,\sigma_b,\sigma_c)$, in $\sigma$, we can make the following claims whose verifications in view of Lemma~\ref{lem:pattern char} are straightforward and left to the reader.
\begin{enumerate}
	\item Both $\sigma_b$ and $\sigma_c$ belong to $\sigma''$, and the image pair $(\hat{\sigma}_b,\hat{\sigma}_c)$ is either $\rP_{i+1}$-active in $\hat{\sigma}$ for $i=3,7$, or $\rP_{i-1}$-active in $\hat{\sigma}$ for $i=4,8$.
  \item $\Phi$ as defined above is an involution on $\fS_n$.
\end{enumerate}

A direct consequence of the claim (1) above is that $(\sigma_a,\sigma_b,\sigma_c)$ forms a refined $\rP_{3,k}$ (resp., $\rP_{4,k}$, $\rP_{7,k}$, $\rP_{8,k}$) pattern in $\sigma$ if and only if $(\hat{\sigma}_{a'},\hat{\sigma}_b,\hat{\sigma}_c)$ forms a $\rP_{4,k}$ (resp., $\rP_{3,k}$, $\rP_{8,k}$, $\rP_{7,k}$) pattern in $\hat{\sigma}$, for every possible value of $k$. Here the existence and uniqueness of $\hat{\sigma}_{a'}$ is guaranteed as shown in the proof of Lemma~\ref{lem:pattern char}. In other words, we have $\mathbf{P_3}(\sigma)=\mathbf{P_4}(\hat{\sigma})$ (resp., $\mathbf{P_4}(\sigma)=\mathbf{P_3}(\hat{\sigma})$, $\mathbf{P_7}(\sigma)=\mathbf{P_8}(\hat{\sigma})$, $\mathbf{P_8}(\sigma)=\mathbf{P_7}(\hat{\sigma})$), as desired.
\end{proof}

\begin{example}\label{p1p2exa}
The following pair of permutations is linked by our involution $\Phi$, and with it one can verify Theorem~\ref{thm:5-tuple}, as well as Conjectures \ref{conj-1}--\ref{conj-4}.
	\begin{align*}
		&\sigma =4 ~6 ~1 ~9 ~2 ~8 ~7 ~5 ~3\quad \stackrel{\Phi}{\longmapsto}\quad \hat{\sigma} =4 ~6 ~1 ~2 ~9 ~3 ~5 ~7 ~8,
	\end{align*}
	\begin{align*}
		&\text{in $\sigma$,}~\rP_{1}=\{128\},~\rP_{2}=\{192,153\},~\rP_{3}=\{128\},~\rP_{4}=\{192,183,173,153\},\\
		&\text{in $\hat{\sigma}$,}~\rP_{2}=\{193\},~\rP_{1}=\{129,478\},~\rP_{4}=\{193\},~\rP_{3}=\{129,138,458,478\},\\
		&\text{in $\sigma$,}~\rP_{5}=\{128\},~\rP_{6}=\{192,253\},~\rP_{7}=\{128\},~\rP_{8}=\{192,283,273,253\},\\
		&\text{in $\hat{\sigma}$,}~\rP_{6}=\{293\},~\rP_{5}=\{129,578\},~\rP_{8}=\{293\},~\rP_{7}=\{129,238,358,578\}.
	\end{align*}
\end{example}

\subsection{Proofs of Conjectures \ref{conj-5} and \ref{conj-6}}
Recall that Conjectures \ref{conj-5} and \ref{conj-6} concern the following four mesh patterns:
\begin{align*}
\rP_{9}=\pattern{scale=1}{3}{1/1,2/2,3/3}{0/1,0/2,1/1,1/2,3/0,3/1,3/2,3/3}\ ,\quad \rP_{10}=\pattern{scale=1}{3}{1/1,2/3,3/2}{0/1,0/2,1/1,1/2,3/0,3/1,3/2,3/3}\ ,\quad \rP_{11} =\pattern{scale=1}{3}{1/1,2/2,3/3}{1/0,0/2,1/1,1/2,3/0,3/1,3/2,3/3}\ ,\quad \rP_{12}=\pattern{scale=1}{3}{1/1,2/3,3/2}{1/0,0/2,1/1,1/2,3/0,3/1,3/2,3/3}\ .
\end{align*}
The shadings in these four patterns all contain the last column. As a result, if $(x,y,z)$ is an occurrence of one of the above four patterns in a given permutation $\sigma\in\fS_n$, then we must have $z=\sigma_n$. Moreover, once $y$ is given, the choice for $x$ is unique. In other words, counting the number of occurrences of these four patterns is equivalent to counting the elligible $y$'s. 

To show the joint symmetric equidistributions of $(\rP_9,\rP_{10})$ and $(\rP_{11},\rP_{12})$, respectively, the main idea is then to swap those $y$'s contained in $\rP_9$-patterns (resp., $\rP_{11}$-patterns) with those $y$'s contained in $\rP_{10}$-patterns (resp., $\rP_{12}$-patterns), by performing certain kind of ``local complementation''. We need the following notion of ``active zone'' to make the above discussion precise.

\begin{Def}
Let $\sigma\in\fS_n$ be a given permutation, we define its {\it active zone}, denoted as $I(\sigma)$, according to its last letter $\sigma_n$. If $\sigma_n=1$ then we set $I(\sigma):=\varnothing$, the empty set; otherwise we define $I(\sigma):=[a,b]$, where $a$ and $b$ are given as follows. 
\begin{itemize}
	\item Let $a:=\sigma_j+1$, where $j:=\min\{i:\sigma_i<\sigma_n\}$.
	\item If $j=1$ then let $b:=n$, otherwise let $b:=\min\{\sigma_i:i<j\}-1$.
\end{itemize}
Note that by definition, $a\le\sigma_n\le b$, so in this case the active zone $I(\sigma)$ is a nonempty interval.
\end{Def}

We are now in a position to define our second involution $\Psi=\Psi_n$ over $\fS_n$, which could be viewed as a complementation restricted to the active zone. Take a permutation $\sigma\in\fS_n$ and suppose its active zone is given by $I(\sigma)=[a,b]$ with $a=\sigma_j+1$. The image of $\sigma$ under $\Psi$, say $\hat{\sigma}=\Psi(\sigma)$, is constructed as follows. For $1\le i\le n$,
\begin{align*}
\hat{\sigma}_i &:=
	\begin{cases}
		\sigma_i & \text{if} \ \sigma_i\notin I(\sigma),\\
		a+b-\sigma_i & \text{if} \ \sigma_i\in I(\sigma).
	\end{cases}
\end{align*}

For the special case when $\sigma_n=1$, we see $I(\sigma)=\varnothing$ and $\hat{\sigma}=\sigma$. While for cases with $\sigma_n>1$, $\hat{\sigma}$ has the same prefix $\hat{\sigma}_1\cdots\hat{\sigma}_j=\sigma_1\cdots\sigma_j$, so we see that $I(\hat{\sigma})=[a,b]=I(\sigma)$, which implies that $\Psi(\Psi(\sigma))=\sigma$. Hence $\Psi$ as defined above is indeed an involution. We shall apply it in the proof the following equidistribution between two quadruples of pattern statistics, which is clearly stronger than the original Conjectures \ref{conj-5} and \ref{conj-6}.

\begin{theorem}\label{thm:quadruple}
For $n\ge 1$, the following identity holds for any permutation $\sigma\in\fS_n$ and its image $\hat{\sigma}:=\Psi(\sigma)$.
\begin{align}\label{id:quadruple}
(\rP_9,\rP_{10},\rP_{11},\rP_{12})\sigma &= (\rP_{10},\rP_{9},\rP_{12},\rP_{11})\hat{\sigma}.
\end{align}
\end{theorem}
\begin{proof}
If $\sigma_n=1$, then $\hat{\sigma}=\sigma$ and $\sigma$ avoids all four patterns $\rP_9$--$\rP_{12}$, so \eqref{id:quadruple} holds trivially. In what follows we assume that $\sigma_n>1$ and $I(\sigma)=[a,b]$. We introduce the following four subsets of $[n]$, each of which consists of all elligible $y$'s that play the role of the middle entry in any occurrence of the corresponding pattern. For $i=9,10,11,12$, let
$$\cP_i^{\mathrm{m}}(\sigma):=\{\sigma_j: \text{$x\sigma_j\sigma_n$ forms a $\rP_i$-pattern in $\sigma$ for some $x$}\}.$$

We observe the following facts regarding these four subsets.
\begin{enumerate}
	\item $\cP_i^{\mathrm{m}}(\sigma)\subseteq I(\sigma)$.
	\item For each $y\in \cP_i^{\mathrm{m}}(\sigma)$, there exists a unique $x$ such that $xy\sigma_n$ forms $\rP_i$-pattern in $\sigma$.
	\item $\cP_9^{\mathrm{m}}(\sigma)=\cP_{11}^{\mathrm{m}}(\sigma)$ and $\cP_{10}^{\mathrm{m}}(\sigma)=\cP_{12}^{\mathrm{m}}(\sigma)$.
	\item Suppose $y=\sigma_j$, then $y\in\cP_9^{\mathrm{m}}(\sigma)$ if and only if 1) $a\le y<\sigma_n$, and 2) for all $k<j$ such that $\sigma_k<\sigma_n$, we must have $\sigma_k<y$.
	\item Suppose $y=\sigma_j$, then $y\in\cP_{10}^{\mathrm{m}}(\sigma)$ if and only if 1) $b\ge y>\sigma_n$, and 2) for all $k<j$ such that $\sigma_k>\sigma_n$, we must have $\sigma_k>y$.
\end{enumerate}

Most of the claims above are straighforward to verify by examining the shadings in the four patterns $\rP_9$--$\rP_{12}$. We only give more details on item (2). If $y=\sigma_j$ belongs to $\cP_{9}^{\mathrm{m}}(\sigma)$ or $\cP_{10}^{\mathrm{m}}(\sigma)$, then $x$ is given by $\max\{\sigma_k:\text{$k<j$ and $\sigma_k<\sigma_n$}\}$; if $y=\sigma_j$ belongs to $\cP_{11}^{\mathrm{m}}(\sigma)$ or $\cP_{12}^{\mathrm{m}}(\sigma)$, then $x=\sigma_k$ where $k$ is given by $\max\{\ell:\text{$\ell<j$ and $\sigma_{\ell}<\sigma_n$}\}$.

With these observations in mind and recall that $\Psi$ is an involution so $\sigma=\Psi(\Psi(\sigma))=\Psi(\hat{\sigma})$, we see that in order to establish \eqref{id:quadruple}, it suffices now to show that
\begin{align*}
|\cP_9^{\mathrm{m}}(\sigma)|=|\cP_{10}^{\mathrm{m}}(\hat{\sigma})|.
\end{align*}

Indeed, the two sets $\cP_9^{\mathrm{m}}(\sigma)$ and $\cP_{10}^{\mathrm{m}}(\hat{\sigma})$ not only share the same cardinality, they are in one-to-one correspondence in the following sence:
$$\sigma_j\in \cP_{9}^{\mathrm{m}}(\sigma)\text{ if and only if } \hat{\sigma}_j\in \cP_{10}^{\mathrm{m}}(\hat{\sigma}).$$
Thanks to the alternative characterizations given in claims (4) and (5) above, and the fact that $\Psi$ restricted to the active zone $I(\sigma)$ is simply the complementation map (hence all inequalities between entries inside the active zone are reversed going from $\sigma$ to $\hat{\sigma}$), the above statement is easily verified.
\end{proof}

\begin{remark}
In view of both observations (2) and (3) in the proof of Theorem~\ref{thm:quadruple}, we note in passing the following relations that hold for any permutation $\sigma\in\fS_n$:
$$\rP_9(\sigma)=|\cP_9^{\mathrm{m}}(\sigma)|=|\cP_{11}^{\mathrm{m}}(\sigma)|=\rP_{11}(\sigma),\text{ and } \rP_{10}(\sigma)=|\cP_{10}^{\mathrm{m}}(\sigma)|=|\cP_{12}^{\mathrm{m}}(\sigma)|=\rP_{12}(\sigma).$$
Consequently, all four mesh patterns $\rP_9$--$\rP_{12}$ are equidistributed over $\fS_n$ for all $n\ge 1$.
\end{remark}

\begin{example}\label{p9p10exa}
We consider the following permutation $\sigma\in\fS_{15}$ and its image $\hat{\sigma}$ under $\Psi$.
	\begin{align*}
		\sigma &=(13) ~(15)~ 4 ~ ({\bf11})~ 2~ {\bf5}~ ({\bf10}) ~ 1 ~ (14) ~ {\bf8} ~ {\bf6} ~ ({\bf12}) ~ 3 ~ {\bf9} ~ {\bf7},\\
		\hat{\sigma} &= (13)~(15)~4 ~{\bf6} ~2 ~({\bf12})~ {\bf7} ~ 1 ~ (14) ~ {\bf9} ~ ({\bf11}) ~ {\bf5} ~ 3 ~ {\bf8} ~ ({\bf10}).
	\end{align*}
Note that $I(\sigma)=I(\hat{\sigma})=[5,12]$, and the entries that fall in the active zone have been boldfaced. One could verify the following sets of occurrences of the four patterns $\rP_9$--$\rP_{12}$ in $\sigma$ and $\hat{\sigma}$, thus confirming the identity \eqref{id:quadruple}.
\begin{align*}
&\text{In $\sigma$:}~ \rP_{9}=\{457,~567\},\quad \rP_{10}=\{4(11)7,~5(10)7,~587\},\\ 
&\phantom{\text{in $\sigma$:}} ~\rP_{11}=\{257,~167\},\quad \rP_{12}=\{4(11)7,~5(10)7,~187\}.\\
&\text{In $\hat{\sigma}$:}~ \rP_{10}=\{6(12)(10),~9(11)(10)\},\quad \rP_{9}=\{46(10),~67(10),~79(10)\},\\
&\phantom{\text{in $\sigma$:}} ~\rP_{12}=\{2(12)(10),~9(11)(10)\},\quad \rP_{11}=\{46(10),~27(10),~19(10)\}.
\end{align*}
\end{example}

\subsection{Two proofs of Conjecture~\ref{conj-7}}
It remains to prove the last conjecture of Lv and Zhang, which concerns the following two mesh patterns:
$$\rP_{13} =\pattern{scale=1}{3}{1/1,2/2,3/3}{1/0,1/1,2/0,2/1,2/2,3/0,3/1,3/2}\ , \quad \rP_{14}=\pattern{scale=1}{3}{1/1,2/3,3/2}{1/0,1/1,2/0,2/1,2/2,3/0,3/1,3/2}\ .$$
We first give the following characterizations of the entries in any occurrence of these two patterns, whose proof should be clear from observing the shadings in the two patterns.
\begin{proposition}\label{p13p14prop}
Suppose $x$, $y$, and $z$ occur (in that order) in a given permutation $\sigma$.
	\begin{enumerate}
		\item $(x,y,z)$ forms a $\rP_{13}$-pattern in $\sigma$, if and only if $x,y,z$ are three consecutive right-to-left minima of $\sigma$.
		\item $(x,y,z)$ forms a $\rP_{14}$-pattern in $\sigma$, if and only if $x,z$ are two consecutive right-to-left minima of $\sigma$, and to the right of $y$, there is only one element, namely $z$, that is less than $y$.
		\item For each triple $(x,y,z)$ that forms either a $\rP_{13}$- or a $\rP_{14}$-pattern of $\sigma$, once the middle entry $y$ is given, $x$ and $z$ are uniquely determined.
	\end{enumerate}
\end{proposition}

Now that we are aware of the importance of the right-to-left minima and ``almost'' right-to-left minima (as those middle entry $y$'s in all occurrences of $\rP_{14}$-patterns), it is beneficial to have an alternative way of expressing a permutation so that these minima could be easily identified. This prompts us to consider the {\it Lehmer code}~\cite{Lem60}, which is a well-known way of encoding permutations; see for instance \cite{Vaj13} for the usefulness of Lehmer code in the study of permutation statistics. Let $\cS_n$ denote the set of {\it subexceedant functions} of length $n$, i.e., sequences $(s_1,\ldots,s_n)$ of $n$ nonnegative integers such that $0\le s_i\le n-i$ for $1\le i\le n$. The version of the Lehmer code required here is the following bijection $L$.
\begin{align*}
L: \fS_n &\to \cS_n \\
\sigma &\mapsto L(\sigma)=(s_1,\ldots,s_n),
\end{align*}
where $s_i=|\{j: j>i, \sigma_j < \sigma_i\}|$.

From the definition of $L$, it is evident that an entry $\sigma_i$ in a permutation $\sigma$ is a right-to-left minimum, if and only if the corresponding entry in the image $L(\sigma)$ is $s_i=0$. With this in mind we introduce the following two sets of indices that characterize the middle entries in any occurrence of the two patterns $\rP_{13}$ and $\rP_{14}$ (see Lemma~\ref{lem:key for conj-7} (1)). Given a subexceedant function $\bs=(s_1,\ldots,s_n)\in\cS_n$, we define two sets associated with it:
\begin{align*}
\cN_0(\bs) &:=\{j: s_j=0, \text{ and there exists $i<j<k$ such that $s_i=s_k=0$}\},\\
\cN_1(\bs) &:=\{j: s_j=1, \text{ and there exists $i<j<k$ such that $s_i=s_k=0$}\},
\end{align*}
and a map from $\cS_n$ to itself:
\begin{align*}
\theta: \cS_n &\to \cS_n \\
\bs=(s_1,\ldots,s_n) &\mapsto \bt=(t_1,\ldots,t_n),
\end{align*}
where for $1\le i\le n$,
\begin{align*}
t_i &:=\begin{cases}
		s_i & ~ \text{if}~ i \notin \cN_0(\bs)\cup \cN_1(\bs),\\ 
		1-s_i & ~ \text{if}~ i \in \cN_0(\bs)\cup \cN_1(\bs).
		\end{cases}
\end{align*}

The following lemma is the key ingredient in our proof of Conjecture~\ref{conj-7}, whose first part is essentially Proposition~\ref{p13p14prop} (3) recast in terms of Lehmer codes.
\begin{lemma}\label{lem:key for conj-7}
\begin{enumerate}
	\item For any permutation $\sigma$ and $\bs=L(\sigma)$, we have 
	\begin{align}\label{id:P and M}
	\rP_{13}(\sigma)=|\cN_0(\bs)|, \text{ and } \rP_{14}(\sigma)=|\cN_1(\bs)|.
	\end{align}
	\item The map $\theta$ as defined above is an involution over $\cS_n$.
\end{enumerate}
\end{lemma}
\begin{proof}
The truth of (1) is based on the following two facts.
\begin{enumerate}[(i)]
	\item $j\in\cN_0(\bs)$ if and only if $\sigma_j$ appears as the middle entry in an occurrence of $\rP_{13}$-pattern in $\sigma$. $j\in\cN_1(\bs)$ if and only if $\sigma_j$ appears as the middle entry in an occurrence of $\rP_{14}$-pattern in $\sigma$.
	\item Each occurrence of pattern $\rP_{13}$ or $\rP_{14}$ is uniquely determined by its middle entry.
\end{enumerate}
To see (i), suppose $j\in\cN_0(\bs)$ and find the largest $i$ and smallest $k$ such that $i<j<k$ and $s_i=s_k=0$, then $(\sigma_i,\sigma_j,\sigma_k)$ forms a $\rP_{13}$-pattern in $\sigma$ according to Proposition~\ref{p13p14prop} (1), the ``if'' part should be clear as well. The case of $j\in\cN_1(\bs)$ is analogous and relies on Proposition~\ref{p13p14prop} (2). Fact (ii) is simply a restatement of Proposition~\ref{p13p14prop} (3).

For (2), it suffices to show that if $\theta(\bs)=\bt$, then
\begin{align}\label{id:N0N1 switch}
\cN_0(\bt)=\cN_1(\bs), \text{ and } \cN_1(\bt)=\cN_0(\bs).
\end{align}
Indeed, when $\bs$ has only one $0$ (equivalently, this means that $\sigma=L^{-1}(\bs)$ ends with $1$), both $\cN_0(\bs)$ and $\cN_1(\bs)$ are empty sets and thus $\bt=\bs$, whence \eqref{id:N0N1 switch} holds true. Otherwise $\bs$ has at least two $0$'s, we see that neither the leftmost $0$ nor the rightmost $0$ belongs to $\cN_0(\bs)\cup \cN_1(\bs)$, which implis that they both stay as $0$ in $\bt$, being still the leftmost and the rightmost $0$ of $\bt$. This in turn guarantees that $i\in\cN_0(\bs)$ (resp., $i\in\cN_1(\bs)$) if and only if $i\in\cN_1(\bt)$ (resp., $i\in\cN_0(\bt)$), which results in \eqref{id:N0N1 switch} as well.
\end{proof}

Now we have all the components needed in our proof of the last conjecture of Lv and Zhang.

\begin{proof}[Proof of Conjecture~\ref{conj-7}]
Thanks to \eqref{id:P and M} and \eqref{id:N0N1 switch}, all it takes now to prove Conjecture~\ref{conj-7}, is to consider the functional composition
$$\Theta:=L^{-1}\circo\theta\circo L,$$
which is an involution over $\fS_n$, such that for any permutation $\sigma\in\fS_n$ and its image $\hat{\sigma}=\Theta(\sigma)$, we have the following two strings of euqations. 
\begin{align*}
\rP_{13}(\sigma)=|\cN_0(L(\sigma))|=|\cN_1(\theta\circo L(\sigma))|=\rP_{14}(L^{-1}\circo\theta\circo L(\sigma))=\rP_{14}(\hat{\sigma}),\\
\rP_{14}(\sigma)=|\cN_1(L(\sigma))|=|\cN_0(\theta\circo L(\sigma))|=\rP_{13}(L^{-1}\circo\theta\circo L(\sigma))=\rP_{13}(\hat{\sigma}).
\end{align*}
Viewing the two ends of the above equations, we confirm Conjecture~\ref{conj-7}.
\end{proof}

\begin{example}\label{p13p14exa}
In this example, given a permutation $\sigma\in\fS_{14}$, we break down the steps to find its image $\hat{\sigma}=\Theta(\sigma)$, and verify the pattern statistics for both $\sigma$ and $\hat{\sigma}$. The letters whose indices fall in the union $\cN_0(L(\sigma))\cup \cN_1(L(\sigma))$ have been boldfaced.
	\begin{align*}
		\sigma = &~3~5~ 2~ 6 ~1~ 4 ~12~ 8~ 9 ~7~ 11~ 14~ 13 ~10\\
		\stackrel{L}{\mapsto} &~(2,~3,~1,~2,~0,~{\bf0},~5,~{\bf1},~{\bf1},~{\bf0},~{\bf1},~2,~{\bf1},~0)\\
		\stackrel{\theta}{\mapsto} &~(2,~3,~1,~2,~0,~{\bf1},~5,~{\bf0},~{\bf0},~{\bf1},~{\bf0},~2,~{\bf0},~0)\\
		\stackrel{L^{-1}}{\mapsto} &~3 ~5 ~2~ 6 ~1~ 7~ 12~ 4~ 8~ 10~ 9~ 14~ 11~ 13=\hat{\sigma}.
	\end{align*}
\begin{align*}
\text{In $\sigma$:}&~\rP_{13}=\{147,~47(10)\},	\quad \rP_{14}=\{487,~497,~7(11)(10),~7(13)(10)\},\\
\text{In $\hat{\sigma}$:}&~\rP_{14}=\{174,~8(10)9\},	\quad \rP_{13}=\{148,~489,~89(11),~9(11)(13)\}.
\end{align*}
\end{example}

If we consider the bivariate enumerator $F_n(s,t):= \sum_{\sigma \in \fS_{n}} s^{\rP_{13}(\sigma)}t^{\rP_{14}(\sigma)}$, then Conjecture~\ref{conj-7} is equivalent to saying that $F_n(s,t)=F_n(t,s)$ for all $n\ge 1$. We have the following recurrence relation for $F_n(s,t)$ which manifestly implies this symmetry. This is an algebraic approach, abeit the proof we present here is still via combinatorial analysis, aligned with the theme of this entire section.

\begin{theorem}\label{recforp13p14}
We have $F_1(s,t)=1$, $F_2(s,t)=2$, and for $n \ge 3$, the following recurrence relation holds.
\begin{align}\label{id:recforp13p14}
	F_n(s,t)=n(n-2)!+(n-2+s+t)(F_{n-1}(s,t)-(n-2)!).
\end{align}
Consequently, $F(s,t)$ is a polynomial in $(s+t)$ and Conjecture~\ref{conj-7} is true.
\end{theorem}
\begin{proof}
The cases $n=1,2$ are clear, so we assume $n\ge 3$ in what follows. We analyze the constant term $F_n(0,0)$ first. Suppose $\sigma\in\fS_n$. We claim the following characterization.
\begin{align}\label{P13P14zero}
\rP_{13}(\sigma)=\rP_{14}(\sigma)=0 \text{ if and only if $\sigma_{n-1}=1$ or $\sigma_n=1$.}
\end{align}
The ``if'' part is clear. For the other direction, we assume that $\sigma_i=1$ for some $1\le i\le n-2$. Now let $\sigma_j:=\min\{\sigma_k: i<k<n\}$ and consider the following two cases.
\begin{enumerate}
	\item $\sigma_j<\sigma_n$. Then $\sigma_j$ is itself a right-to-left minimum, a third one apart from $\sigma_n$ and $\sigma_i=1$, thus we see that $\rP_{13}(\sigma)>0$.
	\item $\sigma_j>\sigma_n$. Then $j\in\cN_1(L(\sigma))$ and from \eqref{id:P and M} we see that $\rP_{14}(\sigma)=|\cN_1(L(\sigma))|>0$.
\end{enumerate}

In $\fS_n$, there are $(n-1)!$ permutations ends with letter $1$, and $(n-1)!$ permutations with letter $1$ in the penultimate position. So \eqref{P13P14zero} leads to the enumeration $F_n(0,0)=2(n-1)!$ for $n\ge 2$. Setting $F_n^*(s,t):=F_n(s,t)-F_n(0,0)$, we could rewrite \eqref{id:recforp13p14} as follows.
\begin{align}\label{id:recF*}
F_n^*(s,t) &= (n-2+s+t)F_{n-1}^*(s,t)+(n-2)!(s+t).
\end{align}

For $n\ge 3$, we denote $\fS_n^*:=\{\sigma\in\fS_n: \sigma^{-1}(1)\le n-2\}$. Then we derive from \eqref{P13P14zero} that $F_n^*(s,t)=\sum_{\sigma\in\fS_n^*}s^{\rP_{13}(\sigma)}t^{\rP_{14}(\sigma)}$. To explain \eqref{id:recF*} combinatorially, we insert letter $n$ into a permutation $\bar{\sigma}\in\fS_{n-1}$ to get a permutation $\sigma\in\fS_n^*$, keeping track of the changes on the numbers of $\rP_{13}$- and $\rP_{14}$-patterns caused by this insertion. It seems that using Lehmer code could again make the explanation clearer. Suppose $L(\bar{\sigma})=\bar{\bs}$, $L(\sigma)=\bs$, and that $n$ is inserted before $\bar{\sigma}_j$, $1\le j\le n$\footnote{Here ``inserted before $\bar{\sigma}_n$'' refers to the case that $n$ is appended to the right end of $\bar{\sigma}$, i.e., after $\bar{\sigma}_{n-1}$.}. If $\bar{\sigma}\in\fS_{n-1}^*$, there are three cases to consider. 
\begin{itemize}
	\item $1\le j\le n-2$, then we see that $s_l=\bar{s}_l$ for $1\le \ell\le j-1$, $s_l=\bar{s}_{l-1}$ for $j+1\le \ell\le n$, and $s_j=n-j\ge 2$. Therefore $\cN_0(\bs)=\cN_0(\bar{\bs})$ and $\cN_1(\bs)=\cN_1(\bar{\bs})$, or equivalently, $\rP_{13}(\sigma)=\rP_{13}(\bar{\sigma})$ and $\rP_{14}(\sigma)=\rP_{14}(\bar{\sigma})$. This explains the term $(n-2)F_{n-1}^*(s,t)$.
	\item $j=n-1$, then $s_{n-1}=1$, whence $\cN_1(\bs)=\cN_1(\bar{\bs})\cup\{n-1\}$ and $|\cN_0(\bs)|=|\cN_0(\bar{\bs})|$, or equivalently, $\rP_{14}(\sigma)=\rP_{14}(\bar{\sigma})+1$ and $\rP_{13}(\sigma)=\rP_{13}(\bar{\sigma})$, which gives rise to the term $tF_{n-1}^*(s,t)$.
	\item $j=n$, then $\sigma_n=n$ is a new right-to-left minimum. More precisely we have $\cN_0(\bs)=\cN_0(\bar{\bs})\cup\{n-1\}$ and $\cN_1(\bs)=\cN_1(\bar{\bs})$, which explains the term $sF_{n-1}^*(s,t)$.
\end{itemize}
Otherwise, we must have $\bar{\sigma}\in\fS_{n-1}\setminus\fS_{n-1}^*$, i.e., $\bar{\sigma}_{n-2}=1$ or $\bar{\sigma}_{n-1}=1$. But the latter case is impossible since there is no way to insert $n$ into such $\bar{\sigma}$ and arrive at $\sigma\in\fS_n^*$. So we have $\bar{\sigma}_{n-2}=1$ (there are $(n-2)!$ such $\bar{\sigma}$'s) and $n$ can only be inserted before $\bar{\sigma}_{n-1}$ or after $\bar{\sigma}_{n-1}$ to guarantee that $\sigma\in\fS_n^*$. The first case gives us $(n-2)!t$ while the latter case accounts for $(n-2)!s$.

We have established \eqref{id:recF*} and the proof is now complete.
\end{proof}

%%%%%%%%%%%%%%%%%%%%%%%%%%%%%%%%%%%%%%%%%%%%%
\section{Two monotone mesh patterns with corner shadings}\label{sec:monotone}
%%%%%%%%%%%%%%%%%%%%%%%%%%%%%%%%%%%%%%%%%%%%%
We begin this section with a proof of Proposition~\ref{prop:2 corners suffice}, then we present three enumerative results for permutations that avoid a pair of monotone (mesh) patterns, one from the ascending family $\rA_k$, one from the descending family $\rD_k$.

\begin{proof}[Proof of Proposition~\ref{prop:2 corners suffice}]
The proofs of \eqref{eq:A=tildeA} and \eqref{eq:D=tildeD} are almost identical, or alternatively, one gets one identity from the other by taking complement. Therefore we show only \eqref{eq:A=tildeA}. The inclusion $\fS_n(\rA_k)\subseteq\fS_n(\widetilde{\rA}_k)$ is clear from the definitions of the avoidance of patterns $\rA_k$ and $\widetilde{\rA}_k$. We aim to show the reverse inclusion. Suppose there exists a permutation $\sigma\in\fS_n(\widetilde{\rA}_k)$ but $\sigma\not\in\fS_n(\rA_k)$. Thus we can find a $k$-tuple of indices $1\le a_1<a_2<\cdots<a_k \le n$ such that
\begin{enumerate}
	\item $\sigma_{a_1}\sigma_{a_2}\cdots\sigma_{a_k}$ forms an occurrence of pattern $\rA_k$ in $\sigma$;
	\item $|a_1-a_k|=\max\{|b_1-b_k|: \sigma_{b_1}\cdots\sigma_{b_k} \text{ forms an occurrence of $\rA_k$ in $\sigma$}\}$.
\end{enumerate}
We claim that $\sigma_{a_1}\sigma_{a_2}\cdots\sigma_{a_k}$ also constitutes an occurrence of pattern $\widetilde{\rA}_k$ in $\sigma$, which leads to a contradiction. Assume not, then either there exists a certain $x\in [a_1-1]$ such that $\sigma_x<\sigma_{a_1}$; or we can find a certain $y\in[a_k+1,n]$ such that $\sigma_y>\sigma_{a_k}$. In the first case, we see that the $k$-tuple $(x,a_2,\ldots,a_k)$ witnesses another occurrence of pattern $\rA_k$ and $|x-a_k|>|a_1-a_k|$, contradicting our choice of $(a_1,\ldots,a_k)$ (condition (2)). For the second case, the $k$-tuple $(a_1,\ldots,a_{k-1},y)$ results in a similar contradiction.

In conclusion, the existence of such $\sigma$ leads to a contradiction so we must have the reverse inclusion $\fS_n(\rA_k)\supseteq\fS_n(\widetilde{\rA}_k)$ and this completes the proof of \eqref{eq:A=tildeA}.
\end{proof}

\begin{proposition}\label{thm:P1P2}
	We have $\fS_1(\rA,\rD)=\fS_1=\{1\}$, and $\fS_n(\rA,\rD)=\varnothing$ for $n\geq 2$.
\end{proposition}

\begin{proof}
	Suppose $n\geq 2$ and there exists a permutation $\sigma\in\fS_n(\rA,\rD)$, then considering the positions of letters $1$ and $n$ in $\pi$ reveals that we have either an occurrence of pattern $\rA$ in $\cdots 1\cdots n\cdots$, or an occurrence of pattern $\rD$ in $\cdots n\cdots 1\cdots$, leading to a contradiction in both cases.
\end{proof}

\begin{theorem}\label{thm:P1P4}
	We have $\fS_1(\rA,\rD_3)=\fS_1=\{1\}$, and for $n\geq 2$, $|\fS_n(\rA,\rD_3)|=n-1$. More precisely,
	$$\fS_n(\rA,\rD_3)=\{a(a+1)\cdots (n-1)n12\cdots (a-1): 2\le a\le n\}.$$ 
\end{theorem}

\begin{proof}
	Assume $n\geq 2$ and take any permutation $\sigma\in \fS_n(\rA, \rD_3)$. Note that in order to avoid pattern $\rA$, $n$ must precedes $1$ in $\sigma$, and $n$ and $1$ must occur consecutively to avoid pattern $\rD_3$. So we can assume $\sigma$ has the decomposition $\sigma=\sigma' n1 \sigma''$. It remains to determine $\sigma'$ and $\sigma''$. We claim the following two observations.
	\begin{enumerate}
		\item The letters in $\sigma'$ are larger than those in $\sigma''$.
		\item The letters in both $\sigma'$ and $\sigma''$ are monotonically increasing.
	\end{enumerate}

	Let $a_1$ be the smallest element in $\sigma'$ and $a_2$ be the largest element in $\sigma''$. Now claim (1) is equivalent to saying that $a_1>a_2$ or at least one of $\sigma'$ and $\sigma''$ is empty. Assume on the contrary that we have $a_1<a_2$. In order to avoid $\rA$, we must find either in $\sigma'$ a letter $b_1$ to the left of $a_1$ such that $b_1>a_2$, or in $\sigma''$ a letter $b_2$ to the right of $a_2$ such that $b_2<a_1$. In the first case the triple $b_1a_1 1$ presents an occurrence of $\rD_3$, while in the second case the triple $na_2b_2$ presents an occurrence of $\rD_3$. So our assumption $a_1<a_2$ cannot hold and the claim (1) is valid.

	To show claim (2), simply note that if there is a descent, say $ab$ with $a>b$, in $\sigma'$, then the triple $ab1$ forms pattern $\rD_3$. Analogously, if there is a descent $ab$ in $\sigma''$, then the triple $nab$ forms pattern $\rD_3$.
	
	Now we have verified the two claims, from which we can deduce that all permutations in $\fS_n(\rA,\rD_3)$ are of the form $a(a+1)\cdots (n-1)n12\cdots (a-1)$ for a certain $2\le a\le n$. Conversely, every such permutation avoids the mesh pattern $\rA$ and the classical pattern $321$, and hence the pattern $\rD_3$ as well. So we establish the desired characterization of $\fS_n(\rA,\rD_3)$ for $n\ge 2$.
\end{proof}

According to the celebrated Erd\"os-Szekeres Theorem~\cite{ES35} (see \cite{Jun24} for a modern narrative on the stories behind this so-called ``happy ending problem''), one sees that $\fS_n(123,321)=\varnothing$ for any $n\ge 5$, which reduces the enumeration to boredom. When shadings are taken into account however, we find the following non-trivial enumerative result concerning the mesh patterns $\rA_3$ and $\rD_3$.

\begin{theorem}\label{thm:P3P4}
	For $n\ge 1$, we have $|\fS_n(\rA_3, \rD_3)|=2^{n-1}$.
\end{theorem}

\begin{proof}
	Set $f(n):=|\fS_n(\rA_3, \rD_3)|$. It is easy to check the statement for $n=1, 2, 3$. Hence we assume $n\geq 4$ and let $\sigma\in \fS_n(\rA_3, \rD_3)$. The proof breaks down to three steps.

	\noindent {\bf Step 1.} Note that $\fS_n(\rA_3, \rD_3)$ is closed under the reversal 
		$$\sigma\to\sigma^r:=\sigma_n\sigma_{n-1}\cdots\sigma_1,$$ 
		so we can assume without loss of generality that $1$ precedes $n$ in $\sigma$. In view of the avoidance of pattern $\rA_3$, $1$ and $n$ must be consecutive in $\sigma$. From now on we assume $1n$ is a factor in $\sigma$ and the total number of such permutations is $f(n)/2$. 
	\smallskip

	\noindent {\bf Step 2.} Assume $n-1$ is to the right of $n$ in $\sigma$, then $n-1$ must be adjacent to $n$. Indeed, suppose there exists a letter $x$ such that $1nx(n-1)$ is a subsequence in $\sigma$. Note that $1x(n-1)$ is an occurrence of $\rA_3$, which is a contradiction. Replacing $n(n-1)$ by $n-1$ in $\sigma$, we obtain a permutation in $\fS_{n-1}(\rA_3, \rD_3)$ in which $1(n-1)$ is a factor. Hence, the number of these permutations considered in Step 2 is given by $f(n-1)/2$.
	\smallskip
		
	\noindent {\bf Step 3.} Let $g(n)$ be the number of permutations $\sigma\in \fS_n(\rA_3, \rD_3)$ which are considered in Step 1 but not in Step 2. I.e., $\sigma$ has the decomposition $\sigma=\sigma^{\prime}(n-1)\sigma^{\prime \prime}1n\sigma^{\prime \prime \prime}$. There are two subcases.

  \noindent {\bf 3-1.} $2\in \sigma^{\prime}(n-1)\sigma^{\prime \prime}$ is to the left of $1$, then $21$ must be a factor of $\sigma$. Indeed, if not, then $2xn$ is an occurrence of $\rA_3$ with $x\neq 2$ being the rightmost letter of $\sigma''$ (or $x=n-1$ in the case of $\sigma''=\varepsilon$). Replacing $21$ by $1$ and reducing all other letters by $1$, we obtain a permutation counted by $g(n-1)$. Note that this operation on $\sigma$ is invertible. That is, any permutation $\tau$ counted by $g(n-1)$ can be turned into a permutation $\sigma$ counted by $g(n)$, by increasing all letters by $1$ and replacing $2$ by $21$. No forbidden patterns can be introduced. Indeed, $2$ and $1$ are not involved together in an occurrence of $\rD_3$ because of $n$, while if $1$ alone is involved in an occurrence of $\rA_3$ or $\rD_3$, $2$ would have been involved (as $1$) in the same occurrence in $\tau$, contradicting with $\tau\in \fS_{n-1}(\rA_3, \rD_3)$. Hence, in case 3-1 we have $g(n-1)$ permutations.

	\noindent {\bf 3-2.} $2\in\sigma^{\prime \prime \prime}$ is to the right of $1$. First note that $2$ must be next to $n$, as otherwise $(n-1)x2$ is an occurrence of $\rD_3$ for some $x\in \sigma^{\prime \prime \prime}$. Similarly, $\sigma^{\prime \prime}=\varepsilon$ or else $(n-1)y2$ is an occurrence of $\rD_3$ for some $y\in\sigma^{\prime \prime}$. Hence, we have $\sigma=\sigma^{\prime}(n-1)1n2\sigma^{\prime \prime \prime}$. Let $\tau$ be the permutation obtained from $\sigma$ by removing $1n$ and decreasing every remaining letter by $1$. Note that $\tau\in\fS_{n-2}(\rA_3, \rD_3)$. Indeed, suppose $\tau$ contains $xyz$ as an occurrence of $\rA_3$ (considering $\rD_3$ is similar and hence is omitted), we consider all possible cases. Note that
		\begin{itemize}
			\item $x$ is to the left of $n-2$ because otherwise either $x=n-2$ (impossible) or $n-2$ is in the shaded area (upper left corner) for $xyz$.
				
			\item $z$ is to the right of $1$ because otherwise either $z=1$ (impossible) or $1$ is in the shaded area (lower right corner) for $xyz$.
		\end{itemize}
	But then it is easy to see that the three letters in $\sigma$ corresponding to $x, y, z$ in $\tau$ also form an occurrence of $\rA_3$, which is impossible. Hence $\tau\in \fS_{n-2}(\rA_3, \rD_3)$ as claimed. We now justify that the map $h: \sigma\rightarrow \tau$ is invertible. Clearly $\sigma$ could be recovered from $\tau$ by increasing every letter by $1$ and inserting $1n$ to separate $n-1$ from $2$. It remains to show that $\sigma$ as constructed above is indeed $\rA_3$- and $\rD_3$-avoiding. Suppose on the contrary that $xyz$ is an occurrence of $\rA_3$ in $\sigma$ ($\rD_3$ can be considered similarly and hence is omitted). Because of letter $2$ (resp., $n-1$), $xyz$ cannot be entirely to the left (resp., right) of $2$ (resp., $n-1$). But then the letters in $\tau$ corresponding to $x, y, z$ would be an occurrence of $\rA_3$, which is a contradiction. Lastly, note that $(n-2)1$ is a factor of $\tau$ so its reversal $\tau^{r}$ is counted in Step 1, and the number of such $\tau$ is given by $f(n-2)/2$. Now $\sigma$ and $\tau$ are in one-to-one correspondence via $h$, so in case 3-2 we have $f(n-2)/2$ permutations.

	Summarizing the cases, we have
	\begin{align*}
	\begin{cases}
		f(n)/2 =f(n-1)/2+g(n);\\
		g(n) =g(n-1)+f(n-2)/2.
	\end{cases}
	\end{align*}
	Solving the system we obtain that $f(n)=2f(n-1)$, which completes the proof through induction.
\end{proof}

%%%%%%%%%%%%%%%%%%%%%%%%%%%%%%%%%%%%%%%%%%%%%%%
\section{132-avoiding permutations that avoid one monotone mesh pattern}\label{sec:132 and one monotone}
%%%%%%%%%%%%%%%%%%%%%%%%%%%%%%%%%%%%%%%%%%%%%%%

The focus in this section is on $132$-avoiding permutations. We enforce on $\fS_n(132)$ the avoidance of a monotone pattern, either from the ascending family $\rA_k$, or from the descending family $\rD_k$, and get enumerative results on the avoiders. Note that the formula given in Theorem \ref{thm:132Ak} actually reduces to the one given in Theorem \ref{thm:132P1} when we set $k=2$. But we still state and prove Theorem \ref{thm:132P1} separately since it is quite typical and makes the proof of Theorem~\ref{thm:132Ak} easier. For all of the proofs, since the starting cases with $n=1,2$ are clearly seen to be true, we always assume that $n\ge 3$ unless otherwise noted.

\begin{theorem}\label{thm:132P1}
	For $n\geq 1$, we have $|\fS_n(132, \rA)|=\left\{\begin{array}{lll}1,& \text{ if }n=1;\\C_n-C_{n-1},&\text{ if }n\geq 2.\end{array}\right.$ 
\end{theorem}

\begin{proof}
	Recall that $|\fS_n(132)|=C_n$, and every permutation $\sigma\in\fS_n(132)$ decomposes nicely, that is, $\sigma=\sigma^{\prime}n\sigma^{\prime \prime}$ where the elements in $\sigma^{\prime}$ are larger than those in $\sigma^{\prime \prime}$. Suppose there are $i$ elements in $\sigma^{\prime}$ and $(n-i-1)$ elements in $\sigma^{\prime \prime}$ for $0\leq i\leq n-2$. The reason for $i\neq n-1$ is that in order to avoid $\rA$, $n$ must precedes $1$ in $\sigma$, since otherwise $(1n)$ will form a pattern $\rA$. Next, we claim that 
	$$\text{$\sigma\in\fS_n(132,\rA)$ if and only if both $\sigma'$ and $\sigma''$ are $132$-avoiding.}$$

	The ``only if'' part is clear since both $\sigma'$ and $\sigma''$ are factors of $\sigma$. Now suppose both $\sigma'$ and $\sigma''\neq \epsilon$ are $132$-avoiding, with all letters in $\sigma'$ larger than all letters in $\sigma''$, we proceed to show that this is enough to ensure that $\sigma=\sigma' n\sigma''\in\fS_n(132,\rA)$. Indeed, these conditions clearly ensure that $\sigma$ avoids $132$. Moreover, since $\sigma''$ is nonempty, every potential $\rA$-pattern pair from $\sigma'n$ will be saved by any letter from $\sigma''$. Likewise, any potential $\rA$-pattern pair from $\sigma''$ will be saved by the presence of $n$ to their left. Hence $\sigma$ avoids $\rA$ as well. 

	The claim is now proven and it gives rise to the following relation, as desired.
	\begin{align*}
		|\fS_n(132, \rA)|=\sum_{i=0}^{n-2}C_iC_{n-i-1}=C_n-C_{n-1},
	\end{align*}
	where we utilize the convolutive recursion for Catalan numbers $C_n=\sum_{i=0}^{n-1}C_{i}C_{n-i-1}$, for $n\ge 1$, and $C_0=1$.
\end{proof}

\begin{theorem}\label{thm:132Ak}
	For $n\geq 1$ and $k\ge 3$, we have 
	$$|\fS_n(132,\rA_k)|=C_{n}-C_{n-1}+|\fS_{n-1}(132,12\cdots(k-1))|.$$
\end{theorem}
\begin{proof}
	Given a permutation $\sigma \in \fS_{n}(132, \rA_k)$, we can express $\sigma$ as $\sigma = \sigma^{\prime} n \sigma^{\prime \prime}$. It is evident that all elements in $\sigma^{\prime}$ must be greater than those in $\sigma^{\prime \prime}$ to avoid the $132$ pattern. We now consider two scenarios.
	\begin{enumerate}
		\item 	If $\sigma^{\prime \prime}\neq\epsilon$, then we claim that it suffices to require that both $\sigma^{\prime}$ and $\sigma^{\prime \prime}$ are $132$-avoiding. Clearly this condition is necessary and it ensures that $\sigma=\sigma'n\sigma''$ is itself $132$-avoiding. To see that it also guarantees that $\sigma$ avoids $\rA_k$, note that any potential $\rA_k$ pattern, i.e., an increasing subsequence of length $k$, must be completely contained in either $\sigma'$ or $\sigma''$ but not in both. If it is contained in $\sigma'$, it is prevented from being a truly $\rA_k$ pattern by any element in $\sigma''$ (recall that $\sigma''\neq\epsilon$); if it is contained in $\sigma''$, it is prevented from being a truly $\rA_k$ pattern by $n$. Consequently, we see that all $\sigma$ in this case are in bijection with the pair $(\sigma',\sigma'')$. Then a similar computation as in the proof of Theorem~\ref{thm:132P1} outputs the count $C_n-C_{n-1}$.
		\item  	If $\sigma^{\prime \prime}=\epsilon$, then $\sigma^{\prime}$ must avoid any increasing subsequence of length $k-1$, since such a subsequence combined with $n$ in the right end will form a $\rA_k$-pattern in $\sigma$. Conversely, as long as $\sigma^{\prime}$ avoids both $132$ and $1\cdots (k-1)$, appending it by $n$ in the end gives us a permutation in $\fS_n(132,\rA_k)$. Therefore, we see that the correspondence between $\sigma=\sigma'n$ and $\sigma'$ is bijective, rendering the count for this case to be $|\fS_{n-1}(132, 1\cdots(k-1))|$.
	\end{enumerate}
	Combining both cases, we obtain the following formula as desired:
	$$|\fS_n(132, \rA_k)| = C_n - C_{n-1} + |\fS_{n-1}(132, 12\cdots(k-1))|.$$
\end{proof}

\begin{remark}
The pattern avoiding class $\fS_n(132,1\cdots k)$ have been enumerated in the literature. Writing $G_k(x):=\sum_{n\ge 0}|\fS_n(132,1\cdots k)|x^n$, we have the following result derived in \cite{CW99} and \cite[Thm.~3.1]{MV00} ; see also \cite[Section~6.1.5]{Kit11}.
$$G_k(x)=\frac{U_{k-1}(\frac{1}{2\sqrt{x}})}{\sqrt{x}U_k(\frac{1}{2\sqrt{x}})},$$
where $U_j$, as defined by $U_j(\cos\theta)=\sin(j+1)\theta/\sin\theta$, is the $j$th Chebyshev polynomial of the second kind.
\end{remark}

The next three enumerative results concern patterns from the decreasing family $\rD_k$. As in the proofs of the previous two theorems, we always assume a $132$-avoiding permutation $\sigma$ decomposes as $\sigma=\sigma'n\sigma''$, where every element in $\sigma'$ is larger than any element in $\sigma''$.

\begin{theorem}\label{thm:132D}
	For $n\geq 1$, we have $|\fS_n(132, \rD)|=C_{n-1}$. 
\end{theorem}

\begin{proof}
	For any $\sigma=\sigma^{\prime}n\sigma^{\prime \prime}\in \fS_n(132, \rD)$, the first thing to notice is that $\sigma^{\prime \prime}=\epsilon$. If not, then $na$ forms a $\rD$-pattern for any element $a$ from $\sigma''$. Hence $\sigma=\sigma^{\prime}n$, where $\sigma^{\prime}\in\fS_{n-1}(132)$. Here only $132$-avoiding is sufficient since any potential $\rD$-pair in $\sigma'$ will be saved by $n$ to their right. Conversely, appending any permutation $\sigma'\in\fS_{n-1}(132)$ with $n$ creates a permutation $\sigma=\sigma'n\in\fS_n(132,\rD)$. So we have $|\fS_n(132, \rD)|=|\fS_{n-1}(132)|=C_{n-1}$.
\end{proof}

% \begin{theorem}\label{thm:132P3}
% 	For $n\geq 1$, we have $|\fS_n(132, \rA_3)|=C_n-C_{n-1}+1$.
% \end{theorem}

% \begin{proof}
% 	It is clear for $n=1, 2$. Similarly we can write $\pi=\pi^{\prime}n\pi^{\prime \prime}\in \fS_n(132, \rA_3)$ where the elements in $\pi^{\prime}$ is larger than those in $\pi^{\prime \prime}$. Now there are two cases (1) $\pi^{\prime \prime}=\emptyset$ and (2) $\pi^{\prime \prime}\neq \emptyset$. For (1), there is only one $\pi=\pi^{\prime}n$, that is, $\pi=(n-1)(n-2)\cdots 1n$. If not, then $(1an)$ will form a pattern $\rA_3$ where positive integer $a\in [2, n-1]$. For (2), we have the same discussion to the proof of Theorem \ref{thm:132P1}, then the number if $C_{n}-C_{n-1}$. Hence we have $|\fS_n(132, \rA_3)|=C_n-C_{n-1}+1$ for $n\geq 3$.
% \end{proof}

\begin{theorem}\label{thm:132D3}
	For $n\geq 1$, we have $|\fS_n(132, \rD_3)|=C_{n-1}+n-1$. 
\end{theorem}

\begin{proof}
	We deal with the two cases (1) $\sigma^{\prime \prime}=\epsilon$ and (2) $\sigma^{\prime \prime}\neq\epsilon$ separately. For (1), a similar argument as in the proof of Theorem~\ref{thm:132D} shows that such $\sigma$ is in bijection with $\sigma'\in\fS_{n-1}(132)$, and the count for this case is $C_{n-1}$. For (2), following the same line of the proof of Theorem~\ref{thm:P1P4}, we see that the elements in $\sigma^{\prime}$ and $\sigma^{\prime \prime}$ are both monotonically increasing, hence the count for this case is $n-1$. Putting together these two cases, we have $|\fS_n(132, \rD_3)|=C_{n-1}+n-1$.
\end{proof}

\begin{theorem}\label{thm:132D4}
	For $n\geq 1$, we have $|\fS_n(132,\rD_4)|=C_{n-1}+\frac{\left(n-1\right)\left(2n^2-7n+12\right)}{6}$.
\end{theorem}

\begin{proof}
Let $\sigma=\sigma^{\prime}n\sigma^{\prime \prime} \in \fS_{n}(132,\rD_4)$. We examine the following three cases:
	 \begin{enumerate}
	 	\item $\sigma^{\prime \prime}$ is empty. The presence of $n$ to the right of $\sigma'$ prevents any decreasing quadruple in $\sigma'$ from being a truly $\rD_4$ pattern of $\sigma$, so $\sigma'$ being $132$-avoiding is sufficient (and necessary). Hence the count for this case is $C_{n-1}$.
	 	\item If $\sigma^{\prime \prime}$ is nonempty but $\sigma^{\prime}$ is empty, then any decreasing triple $xyz$ in $\sigma''$ jointly with $n$ will form a $\rD_4$ pattern in $\sigma$, so $\sigma''$ must be $321$-avoiding (in addition to being $132$-avoiding). Conversely, any permutation $\sigma''\in\fS_{n-1}(132,321)$ with $n$ being appended at the beginning creates a permutation $\sigma=n\sigma''\in\fS_n(132,\rD_4)$. So such $\sigma$ is in bijection with $\sigma''\in\fS_{n-1}(132,321)$. The cardinality $|\fS_{n-1}\left(132,321\right)|$ is known to be $\binom{n-1}{2}+1$; see for instance \cite[Prop.~11]{SS85}. 
	 	\item If neither $\sigma^{\prime}$ nor $\sigma^{\prime \prime}$ is empty, then to avoid the $\rD_4$ pattern, one of these two subpermutations must be in increasing order. Suppose not, then a descent pair $wx$ in $\sigma^{\prime}$ combined with a descent pair $yz$ in $\sigma^{\prime \prime}$ would create an occurrence of the $\rD_4$ pattern. We first assume that $\sigma^{\prime}$ is an increasing subpermutation, then $\sigma^{\prime \prime}$ must avoid both $132$ and $321$ (classical) patterns, resulting in $\sum_{i=1}^{n-2}(\binom{i}{2} + 1)$ permutations. In the same vein, when $\sigma''$ is increasing, we must have $\sigma'$ being simultaneously $132$- and $321$-avoiding, giving rise to another $\sum_{i=1}^{n-2}(\binom{i}{2} + 1)$. Finally, we have to subtract $n-2$ from the total count though, since these permutations $a(a+1)\cdots n12\cdots (a-1)$, $2\le a\le n-1$, have been double-counted.
	 	\end{enumerate}

Combining all cases and simplifying the expression 
	$$C_{n-1} + 2\sum_{i=1}^{n-2}\left(\binom{i}{2} + 1\right) - (n-2) + \binom{n-1}{2} + 1$$ 
yields the desired result. 
\end{proof}

%%%%%%%%%%%%%%%%%%%%%%%%%%%%%%%%
\section{Concluding remarks}\label{sec:concluding_remarks}
%%%%%%%%%%%%%%%%%%%%%%%%%%%%%%%%

Employing combinatorial approach, we verify in this paper all fourteen jointly equidistributed pairs of mesh patterns listed in Tabel 11 from the paper of Lv and Zhang~\cite{LZ25}. The additional conjectures in \cite{LZ25} concerning equidistributions between different pairs are still awaiting proof. Moreover, we would like to make two further remarks that may stimulate future research.

Viewing the enumeration result in Theorem~\ref{thm:132Ak} for any $k\ge 3$, it is natural to seek for a comparable result when $\rA_k$ is replaced by $\rD_k$, so as to generalize the results in Theorems \ref{thm:132D}, \ref{thm:132D3}, and \ref{thm:132D4}.

Furthermore, as alluded to in the Introduction, besides enumerating pattern-avoiding permutations, one can treat the number of occurrences of certain (mesh) pattern as a permutation statistic and study its distribution. The following result is an initial example in this direction. Let $S_n(t):=\sum_{\sigma\in \fS_n(132)}t^{\rA(\sigma)}$ and recall the generating function of Catalan numbers $C(x):=\sum_{n\geq 0}C_nx^n$, which satisfies the quadratic equation $$xC(x)^2-C(x)+1=0.$$
% Let  be the generating function for the Catalan numbers, which is known to satisfy the quadratic equation 
% {\SF The following result should be a refinement of the above quadratic equation, check the derivation.}

\begin{theorem}\label{thm:dis_P1in132}
The generating function $S(x, t):=\sum_{n\geq 0}S_n(t)x^n$ satisfies the following equation:
\begin{align}\label{eq:dis_P1in132}
	S(x, t)=(1-x)C(x)+xS(xt, t).
\end{align}
\end{theorem}

\begin{proof}
	Given a permutation $\sigma=\sigma^{\prime}n\sigma^{\prime \prime}\in \fS_{n}(132)$, $n\ge 2$, we know from previous proofs that $\sigma\in\fS_n(132,\rA)$ if and only if $\sigma''\neq\epsilon$. These permutations correspond to the constant term $S_n(0)$ and are counted by $C_n-C_{n-1}$ (see Theorem~\ref{thm:132P1}).

	Otherwise $\sigma''=\epsilon$, and we see that appending $n$ at the end of $\sigma'$ creates $n-1$ new $\rA$-patterns (in $1n,2n,\ldots,(n-1)n$), and all existing $\rA$-patterns stay as $\rA$-patterns. This discussion gives us
	\begin{align*}
		S_n(t)=(C_n-C_{n-1})t^0+t^{n-1}S_{n-1}(t).
	\end{align*}
	Multiplying both sides by $x^n$ and summing over $n\geq 1$, we have
	\begin{align*}
		\sum_{n\geq 1}S_n(t)x^n=\sum_{n\geq 1}C_nx^n-\sum_{n\geq 1}C_{n-1}x^n+\sum_{n\geq 1}S_{n-1}(t)t^{n-1}x^n,
	\end{align*}
	which simplifies to \eqref{eq:dis_P1in132}.
\end{proof}

\section*{Acknowledgements}
The authors are grateful to Shuzhen Lv for providing data relevant to our research. Shishuo Fu was partially supported by the National Natural Science Foundation of China grants 12171059 and 12371336.

\end{document}